\documentclass{amsart}
\usepackage{amsfonts}
\usepackage{color}
\def\R{{\mathbb {R}}}
\def\N{{\mathbb {N}}}

\def\A{{\mathcal{A}}}

\def\E{{\mathcal{E}}}
\def\K{{\mathcal{K}}}
\def\F{{\mathcal{F}}}

\def\O{{\Omega}}
\def\lam{\lambda}
\def\vp{\varphi}

\def\ve{\varepsilon}

\def\cp{\operatorname {\text{cap}}}

\newtheorem{teo}{Theorem}[section]
\newtheorem{lema}[teo]{Lemma}
\newtheorem{prop}[teo]{Proposition}

\theoremstyle{remark}
\newtheorem{remark}[teo]{Remark}

\theoremstyle{definition}
\newtheorem{defi}[teo]{Definition}

\numberwithin{equation}{section}

\begin{document}
	
\title{Shape optimization problems for nonlocal operators}
\author[J. Fern\'andez Bonder, A. Ritorto and A.M. Salort]{Juli\'an Fern\'andez Bonder, Antonella Ritorto and Ariel Martin Salort}

\address{Departamento de Matem\'atica, FCEN -- Universidad de Buenos Aires and IMAS -- CONICET, Buenos Aires, Argentina}
\email[J. Fern\'andez Bonder]{jfbonder@dm.uba.ar}
\urladdr[J. Fern\'andez Bonder]{http://mate.dm.uba.ar/~jfbonder}
	
\email[A. Ritorto]{aritorto@dm.uba.ar}
	
\email[A.M. Salort]{asalort@dm.uba.ar}
\urladdr[A.M. Salort]{http://mate.dm.uba.ar/~asalort}
	
%35R11 Fractional partial differential equations
%49Q10  	Optimization of shapes other than minimal surfaces
\subjclass[2010]{35R11, 49Q10}
	
\keywords{Fractional partial differential equations, shape optimization}

\begin{abstract}
In this work we study a general shape optimization problem where the state equation is given in terms of a nonlocal operator. Examples of the problems considered are monotone combinations of fractional eigenvalues. Moreover, we also analyze the transition from nonlocal to local state equations.
\end{abstract}
	
\maketitle

%%%%%%%%%%%%%%%%%%%%%%%%%%%%%
%%%%%%%%%%%%%%%%%%%%%%%%%%%%%
%% 
%% INTRODUCCION
%%
%%%%%%%%%%%%%%%%%%%%%%%%%%%%%	
%%%%%%%%%%%%%%%%%%%%%%%%%%%%%

\section{Introduction}

In this article, we consider shape optimization problems that in the most general form can be stated as follows: Given a {\em cost functional} $F$, and a class of {\em admissible domains} $\A$, solve the minimization problem
\begin{equation}\label{min}
\min_{A\in\A} F(A).
\end{equation}

These types of problems have been extensively considered, and they arise in many fields and in many applications. The literature is very wide, from the classical cases of isoperimetrical problems to the most recent applications including elasticity and spectral optimization. Only to mention some references, we refer   the reader to the books of Allaire \cite{Allaire-book}, Bucur and Buttazzo \cite{Bucur-Buttazzo}, Henrot \cite{Henrot-book}, Pironneau \cite{Pironneau-book} and Soko{\l}owski and  Zol{\'e}sio \cite{Sokolowski-Solesio-book}, where a huge amount of shape optimization problems are tackled.

In most of the existing references, the cost functional $F$ is given in terms of a function $u_A$ which is the solution of a {\em state equation} to be solved on $A$ of the form
$$
F(A) = \int_A j(\nabla u_A, u_A, x)\, dx.
$$
Typically, this state equation is an elliptic PDE. 

In recent years there has been an increasing amount of interest in nonlocal problems due to several interesting applications that include some physical models \cite{DGLZ, Eringen, Giacomin-Lebowitz, Laskin, Metzler-Klafter, Zhou-Du}, finance \cite{Akgiray-Booth, Levendorski, Schoutens}, fluid dynamics \cite{Constantin}, ecology \cite{Humphries, Massaccesi-Valdinoci, Reynolds-Rhodes} and image processing \cite{Gilboa-Osher}.

However, there are only a handful of results of shape optimization problems of the form \eqref{min} where the state equation involves a nonlocal operator instead of an elliptic PDE.

For instance, in \cite{Sire-Vazquez-Volzone}, the authors extend the well-known Faber-Krahn inequality to the fractional case and as a simple corollary, they solve problem \eqref{min} in the case where $F(A)=\lambda_1^s(A)$ where $\lambda_1^s(A)$ is the first eigenvalue of the Dirichlet fractional laplacian and the class $\A$ is the class of open sets of fixed measure. (See next section for precise definitions).

In \cite{Brasco-Parini} the authors consider again the class $\A$ of open sets of fixed measure and $F(A)=\lambda_2^s(A)$ and prove that problem \eqref{min} does not have a solution. In fact, a minimization sequence of domains consists of a sequence of balls of the same measure where the distance of the centers diverges.

Finally, in \cite{Bonder-Spedaletti}, the authors take the class $\A$ of measurable sets of fixed measure contained in a fixed open set $\Omega$ and the cost functional $F(A) = \lambda_1^s(\Omega\setminus A)$ where in this case, $\lambda_1^s(\Omega\setminus A)$ is the first eigenvalue of the fractional laplacian with Dirichlet condition on $A$ and Neumann condition in $\R^n\setminus\Omega$.

For other recent shape optimization problems where the state equation is nonlocal, see  \cite{Burchard-Choksi-Topaloglu, Dalibard-Gerard, Knupfer-Muratov, Knupfer-Muratov2, Qiu-Chong-Zhou}, and references therein.

\medskip

The purpose of this article is to consider the general minimization problem \eqref{min} for general costs funcions $F$ under some natural assumptions that includes the particular cases mentioned above. These natural assumptions are similar to those considered in \cite{Buttazzo-DalMaso} where the authors addressed this problem when the state equation is given in terms of an elliptic PDE. Roughly speaking, these assumptions are: 
\begin{itemize}
\item monotonicity with respect to the inclusion and
\item lower semicontinuity with respect to a suitable defined notion of convergence of domains.
\end{itemize}

Observe that the results of \cite{Brasco-Parini} put a restriction on the classes of admissible domains that one needs to consider if you want to obtain a positive result. So, in the spirit of \cite{Buttazzo-DalMaso} we restrict ourselves to the class $\A$ of open sets of fixed measure that are contained in a fixed box $Q\subset \R^n$.

Under these conditions, we are able to recover the results of \cite{Buttazzo-DalMaso} in the fractional setting and, moreover, we analyze the transition from the fractional case to the classical elliptic PDE case proving convergence of the minima and of the optimal shapes.

\section{Setting of the problem}

\subsection{Some preliminaries and notation}
Given $s\in(0,1)$ we consider the fractional laplacian, that for smooth functions $u$ is defined as
\begin{align*}
(-\Delta)^s u(x) &:= c(n,s)\mbox{p.v.}\int_{\R^n} \frac{u(x)-u(y)}{|x-y|^{n+2s}} \, dy\\
&=-\frac{c(n,s)}{2}\int_{\R^n} \frac{u(x+z) - 2u(x) + u(x-z)}{|z|^{n+2s}}\, dz. 
\end{align*}
where $c(n,s):= ( \int_{\R^n} \frac{1-\cos\zeta_1}{|\zeta|^{n+2s}}d\zeta )^{-1}$ is a normalization constant. 

The constant $c(n,s)$ is chosen in such a way that the following identity holds,
$$
(-\Delta)^s u = \F^{-1}(|\xi|^{2s}\F(u)),
$$
for $u$ in the Schwarz class of rapidly decreasing and infinitely differentiable functions, where $\F$ denotes the Fourier transform. See \cite[Proposition 3.3]{DiNezza-Palatucci-Valdinoci}.

The natural functional setting for this operator is the fractional Sobolev space $H^s(\R^n)$ defined as
\begin{align*}
H^s(\R^n)&:=\left\{u\in L^2(\R^n) \colon \frac{u(x)-u(y)}{|x-y|^{\frac{n}{2}+s}}\in L^2(\R^n \times \R^n) \right\}\\
&= \left\{ u\in L^2(\R^n)\colon |\xi|^2\F(u)\in L^2(\R^n)\right\}
\end{align*}
which is a Banach space endowed with the norm $\|u\|^2_s:= \|u\|_2^2 + [u]^2_s $,
where the term
$$
	[u]^2_s:=\iint_{\R^n \times \R^n} {\frac{|u(x)-u(y)|^2}{|x-y|^{n+2s}} \, dxdy} 
$$
is the so-called Gagliardo semi-norm of $u$.

Due to the nonlocal nature of the operator $(-\Delta)^s$, when dealing with Dirichlet-type problems in a open bounded set $\Omega\subset \R^n$, it is necessary to contemplate the ``boundary condition" not only on $\partial\Omega$ but in the whole $\R^n\setminus \Omega$. The natural space to work with is denoted as $H^s_0(\Omega)$ and it is defined by the closure of $C_c^\infty(\Omega)$ in the norm $\| \cdot \|_s$. When $\Omega$ is a Lipschitz domain, $H^s_0(\Omega)$ coincides with the space of functions vanishing outside $\Omega$, i.e., 
$$
	H^s_0(\Omega)=\{u\in H^s(\R^n)\colon u=0 \mbox{ in } \R^n\setminus \Omega\}.
$$
%It will be useful to denote, for $u,v\in H^s_0(\Omega)$
%$$
%	\mathcal{E}(u,v)= \frac{c(n,s)}{2}\iint_{\R^n \times \R^n} {\frac{(u(x)-u(y))(v(x)-v(y))}{|x-y|^{n+2s}} \, dxdy}.
%$$

Aimed at our purposes in this paper, it is suitable to analyze the the behavior of the normalization constant $c(n,s)$ as $s\uparrow 1$. In \cite{Stein}, E. Stein  studied the relation between negative powers of the Laplace operator and Riesz potentials. In this context  it is proved that 
\begin{equation} \label{asinto}
\lim_{s\uparrow 1} \frac{c(n,s)}{1-s}=\frac{4n}{\omega_{n-1}},
\end{equation}
where $\omega_{n-1}$ denotes the $n-1$-dimensional measure of the unit sphere $S^{n-1}$.
That election of the constant is consistent in order to recover the usual laplacian  in the sense that
\begin{equation} \label{conv.lap}
	\lim_{s\uparrow 1} (-\Delta  )^s	u =-\Delta u \quad \forall u\in C_c^\infty(\R^n).
\end{equation}
For a direct proof of these facts, we refer to the article \cite{DiNezza-Palatucci-Valdinoci}.

Moreover, in \cite[Remark 4.3]{DiNezza-Palatucci-Valdinoci} is shown that
\begin{equation} \label{bbm2}
	\lim_{s\uparrow 1} \frac{c(n,s)}{2}[u]_s^2 = \|\nabla u\|_2^2.
\end{equation}

\subsection{Statements of the main results}

We begin with some definitions.

\begin{defi}
Let $\Omega\subset \R^n$ be an open set. Given $A\subset \Omega$, for any $0<s<1$, we define the Gagliardo $s-$capacity of $A$ relative to $\Omega$ as
$$
\cp_s(A,\O)= \inf \left\{ [u]^2_s \colon u\in C^\infty_c(\Omega),\ u\ge 0, \ A\subset \{u\geq 1\}^{\circ} \right\}.
$$

In this context, we say that a subset $A$ of $\O$ is a {\em $s$-quasi open} set if there exists a decreasing sequence $\{\omega_k\}_{k\in  \N}$ of open subsets of $\O$ such that $\cp_s(\omega_k,\O) \to 0$, as $k\to+\infty$, and $A\cup \omega_k$ is an open set for all $k\in \N$. 

We denote by $\A_s(\Omega)$ the class of all $s-$quasi open subsets of $\Omega$. 

In the case $s=1$ the definitions are completely analogous with $\|\nabla u\|_2$ instead of $[u]_s^2$.
\end{defi}

\begin{remark}
From H\"older's inequality is easy to see that $\mathcal{A}_s(\Omega) \subset \mathcal{A}_t(\Omega)$ when $0<t<s\le1$.
\end{remark}

For further properties of the $s-$capacity we refer the reader, for instance, to \cite{Shi-Xiao}.

\medskip

Given $A\in \A_s(\Omega)$, we denote by $u^s_A\in H^s_0(A)$ the unique (weak) solution to
\begin{equation}\label{uas}
(-\Delta)^s u^s_A =1 \quad \mbox{ in } A, \qquad u^s_A=0 \quad \mbox{ in }  \R^n \setminus A.
\end{equation}
With this notation, we define the following notion of set convergence.
\begin{defi}
Let $\{A_k\}_{k\in\N}\subset \A_s(\Omega)$ and $A\in \A_s(\Omega)$. We say that $A_k\stackrel{\gamma_s}{\to} A$ if $u_{A_k}^s\to u_A^s$ strongly in $L^2(\Omega)$.
\end{defi}

\begin{remark}
This is the fractional version of the $\gamma-$convergence of sets defined in \cite{Buttazzo-DalMaso}.
\end{remark}

Now, take $0<s<1$ be fixed and let $F_s\colon \A_s(\Omega)\to \R$ be such that
\begin{enumerate}
\item[($H^s_1$)] \label{H1}$F_s$ is lower semicontinuous with respect to the $\gamma_s-$convergence; that is,  
$$
A_k\stackrel{\gamma_s}{\to} A \quad\text{implies}\quad F_s(A)\le \liminf_{k\to\infty} F_s(A_k).
$$

\item[($H^s_2$)]\label{H2} $F_s$ is decreasing with respect to set inclusion; that is  $F_s(A)\geq F_s(B)$ whenever $A\subset B$.
\end{enumerate}

So, the problem that we address in this paper is the following:
\begin{equation}\label{problem}
	\min \{ F_s(A) \colon A\in \A_s(\Omega),\ |A|\le c\},
\end{equation}
where $F_s$ satisfies ($H_1^s$)--($H_2^s$).

Following the same approach and ideas of \cite{Buttazzo-DalMaso}, problem \eqref{problem} can be analyzed and that is the content of our first result. 
\begin{teo}\label{main.s}
Let $0<s<1$ be fixed and $\Omega\subset \R^n$ be open and bounded. Let $F_s:\mathcal{A}_s(\Omega)\to \R$ be such that {\em($H_1^s$)} and {\em ($H_2^s$)} are satisfied.

Then, for every $0< c< |\Omega|$, problem \eqref{problem} has a solution.
\end{teo}

As we mentioned, the proof of Theorem \ref{main.s} follows the ideas developed in \cite{Buttazzo-DalMaso} and that is carried out in Section \ref{sec.teo1}.

Next, we want to analyze the behavior of this minimum problems and its minimizers when $s\uparrow 1$.

In order to perform such analysis we need to assume some asymptotic behavior on the cost functionals $F_s$. In order to do this, we need to define a notion of convergence for sets when $s$ varies.
\begin{defi}
Let $0<s_k\uparrow 1$ and let $A_k\in \A_{s_k}(\Omega)$ and $A\in \A_1(\Omega)$. We say that $A_k\stackrel{\gamma}{\to} A$ if $u_{A_k}^{s_k}\to u_A^1$ strongly in $L^2(\Omega)$.
\end{defi}

\begin{remark}
Observe that the notion of $\gamma-$convergence of sets given in \cite{Buttazzo-DalMaso} is denoted in this paper by $\gamma_1-$convergence. This should not cause any confusion.
\end{remark}

Now we can give the assumptions of the functionals $F_s$:
\begin{enumerate}
\item[($H_1$)] Continuity with respect to $A$; that is, if $A\in \A_1(\Omega)$, then
$$ 
F_1(A) = \lim_{s\uparrow 1} F_s(A).
$$	
\item[($H_2$)] 	Liminf inequality; that is, for every $0<s_k\uparrow 1$ and $A_k\stackrel{\gamma}{\to}A$, then 
$$
F_1(A) \leq \liminf_{k\to\infty} F_{s_k}(A_k),
$$
\end{enumerate}

Under these assumptions, we obtain the following result.
\begin{teo}\label{main}
For any $0<s\le 1$, let $F_s\colon \A_s(\Omega)\to \R$ be such that {\em ($H^s_1$)} and {\em ($H^s_2$)} are satisfied. Assume moreover that {\em ($H_1$)} and {\em ($H_2$)} are satisfied.

Then
$$
\min\left\{ F_1(A)\colon A\in \A_1(\Omega),\ |A|\le c\right\} = \lim_{s\uparrow 1} \min\left\{ F_s(A)\colon A\in \A_s(\Omega),\ |A|\le c\right\} 
$$
and, moreover, if $A_s\in \A_s(\Omega)$ is a minimizer for \eqref{problem}, then there exists a sequence $0<s_k\uparrow 1$, sets $\tilde{A}_{s_k}\supset A_{s_k}$ and a set $A_1\in \A_1(\Omega)$ such that $\tilde{A}_{s_k}\stackrel{\gamma}{\to}A_1$ and $A_1$ is a minimizer for \eqref{problem} with $s=1$.
\end{teo}

The proof of Theorem \ref{main} is carried out in Section \ref{sec.teo2} and also uses ideas developed in \cite{Buttazzo-DalMaso}. However, in this case nontrivial modifications need to be made in order to consider the varying spaces where one is working.

\subsection{Examples}

Let first establish some notations. Given a bounded domain $A\in \A_s(\Omega)$, consider the problem
\begin{equation} \label{eq}
	(-\Delta)^s u= \lambda^s u \quad \textrm{ in } A, \qquad u\in H_0^s(A)
\end{equation}
where $\lam^s\in \R$ is the eigenvalue parameter. It is well-known that there exists a discrete sequence $\{\lam_k^s(A)\}_{k\in\N}$ of positive eigenvalues of \eqref{eq} approaching $+\infty$ whose corresponding eigenfunctions $\{u_k^s\}_{k\in\N}$ form an orthogonal basis in $L^2(A)$. Moreover, the following variational characterization holds for the eigenvalues
\begin{equation} \label{variac}
	\lam_k^s(A)=\min_{u\perp W_{k-1} }\frac{c(n,s)}{2}\frac{[u]^2_s}{\|u\|_2^2},
\end{equation}
where $W_k$ is the space spanned by the first $k$ eigenfunctions $u_1^s,\ldots, u_k^s$.

Functions $F_s$ satisfying hypothesis ($H_1^s$) and ($H_2^s$) include a large family of examples. For instance, if we consider the application $A\mapsto \lam_k^s(A)$. Theorem \ref{main.s} claims that for every $k\in\N$ and $0< c< |\Omega|$, the minimum
$$
		\min\{ \lam_k^s(A)\colon A\in \A_s(\Omega), \, |A|\le c\}
$$
is achieved. More generally, the minimum
$$
		\min\{ \Phi_s(\lam_{k_1}^s(A),\dots,\lam_{k_N}^s(A))\colon A\in \A_s(\Omega), \, |A|\le c\}
$$
is achieved, where $\Phi_s\colon  \R^N\to \bar \R$, $N\in\N$ is increasing in each coordinate and lower semicontinuous.

Moreover, if $\Phi_s(t_1,\dots,t_N)\to \Phi_1(t_1,\dots,t_N)$ for every $(t_1,\dots,t_N)\in\R^N$ and
$$
\Phi_1(t_1,\dots,t_N)\le \liminf_{k\to\infty} \Phi_{s_k}(t_1^k,\dots,t_N^k),
$$
for every $(t_1^k,\dots,t_N^k)\to (t_1,\dots,t_N)$, then Theorem \ref{main} together with the result of \cite{BPS} imply that 
\begin{align*}
\min\{ \Phi_1&(\lam_{k_1}(A),\dots,\lam_{k_N}(A))\colon A\in \A_1(\Omega), \, |A|\le c\}\\
& = \lim_{s\uparrow 1} \min\{ \Phi_s(\lam_{k_1}^s(A),\dots,\lam_{k_N}^s(A))\colon A\in \A_s(\Omega), \, |A|\le c\}.
\end{align*}

\section{Proof of Theorem \ref{main.s}} \label{sec.teo1}

This section is devoted to prove Theorem \ref{main.s}. The arguments follow essentially the lines of \cite{Buttazzo-DalMaso} with some modifications for the nonlocal setting.

The sketch of the argument is as follows: Given $A\in\A_s(\Omega)$, we first prove that $u_A^s$ is the solution to
\begin{equation}\label{ecA}
\max\{w\in H^s_0(\Omega)\colon w\le 0 \text{ in } \R^n\setminus A,\ (-\Delta)^s w\le 1 \text{ in }\Omega\}.
\end{equation}
Moreover, $u_A^s$ belongs to the convex closed set $\mathcal{K}_s$ defined as
\begin{equation} \label{Ks}
\mathcal{K}_s= \{w \in H^s_0(\Omega) \colon w\geq0, (-\Delta)^s w\leq 1 \textrm{ in } \Omega\}.
\end{equation}
It will be convenient to also consider $\K_1$ defined as in \eqref{Ks} with $s=1$ where $(-\Delta)^1 = -\Delta$.

Finally, one defines a functional $G_s$ on $\mathcal{K}_s$ satisfying that
\begin{itemize}
	\item[($G_1$)] $G_s$ is decreasing on $\mathcal{K}_s$,
	\item[($G_2$)] $G_s$ is semicontinuous on $\mathcal{K}_s$ with respect to the strong topology on $L^2(\Omega)$, 
	\item[($G_3$)] $G_s(u_A^s)=F_s(A)$ for every $A\in\A_s(\Omega)$,
\end{itemize}	
to conclude that the   problem
\begin{equation}\label{minG}
	\min\{ G_s(w)\colon w\in\mathcal{K}_s, \, |\{w>0\}|\leq c \}
\end{equation}
has a solution $w_0$. If we denote by $A_0=\{w_0>0\}$, then $u_{A_0}^s$ is also a minimum point of $G_s$ over the whole $\mathcal{K}_s$ subject to the condition $|\{w>0\}|\leq c$ and hence, $A_0$ is a minimizer for $F_s$ in $\A_s(\Omega)$ subject to the condition $|A|\le c$.

\bigskip

We start by proving \eqref{ecA}. Let us define 
$$
K_A=\{w \in H^s_0(\O)\colon w\leq 0 \textrm{ in } \R^n \setminus A\},
$$
and $w_A\in K_A$ the (unique) minimizer of
$$
I_s\colon K_A\to \R,\qquad I_s(w) = \frac{c(n,s)}{2} [w]_s^2 - \int_\O w\, dx.
$$
Observe that, by Stampacchia's Theorem, $w_A$ is characterized by the variational inequality 
\begin{equation} \label{viuA}
	\E(w_A,v-w_A)\geq \int_{\O}{(v-w_A) \, dx} \quad \forall v \in K_A,
\end{equation}
where we denote
\begin{equation}\label{E.def}
\E(u,v) := c(n,s)\iint_{\R^n\times\R^n} \frac{(u(x)-u(y))(v(x)-v(y))}{|x-y|^{n+2s}}\, dxdy.
\end{equation}

Next, we prove that both functions $u_A^s$ and $w_A$ agree.

\begin{lema}	\label{igualdad}
With the previous notation we have that $w_A=u_A^s$.
\end{lema}

\begin{proof}
The proof is standard. Take $w_A^+$ as test function in the variational inequality \eqref{viuA} and obtain
\begin{align*}
0\le \int_\O w_A^-\, dx &\le \E(w_A, w_A^-)\\
&\le -c(n,s)\iint_{\{w_A\le 0\}\times\{w_A\le 0\}} \frac{(w_A^-(x)-w_A^-(y))^2}{|x-y|^{n+2s}}\, dxdy.
\end{align*}
From this inequality one easily conclude that $w_A^-=0$ and so, since $w_A\in K_A$, $w_A\in H^s_0(A)$.

Therefore, since $u_A^s$ is the unique minimum of $I_s$ over $H^s_0(A)$ and, since also $u_A^s\in K_A$, $I_s(w_A)\le I_s(u_A^s)$ the lemma follows.
\end{proof}

Using the previous lemma we prove the following properties on $u_A^s$.
\begin{lema} \label{teouA} 
	With the above notations, $u_A^s\geq 0$ on $\Omega$. Moreover, $u_A^s$ is the solution to \eqref{ecA}.
\end{lema}

\begin{proof}
	Given $v\in H^s_0(\O)$ such that $v\geq0$, we have that $-v\in K_A$. Using it as a test function in \eqref{viuA} we obtain that
	\begin{align*}
	\E(u_A^s, -v-u_A^s) =  -c(n,s)[u_A^s]^2_s-\E(u_A^s,v) \geq -\int_{\O}{v \, dx}-\int_{\O}{u_A^s \, dx}.
	\end{align*}
	Using that  $(-\Delta)^s u_A^s=1$ in $A$, the last inequality reads as
	$$
	 \E(u_A^s,v)\leq \int_{\O} v \, dx.
	$$
	Since $v\in H^s_0(\Omega)$ is nonnegative but otherwise arbitrary,  we get that $(-\Delta)^s u_A^s \leq 1$ in $\O$.

Finally, if $w\le 0$ in $\R^n\setminus A$ and $(-\Delta)^s w\le 1$ in $\Omega$, then
$$
(-\Delta)^s w\le (-\Delta)^s u_A^s \text{ in } A \quad \text{ and }\quad w\le u_A^s \text{ in } \R^n\setminus A.
$$
Hence, by comparison, $w\le u_A^s$ in $\R^n$.
\end{proof}

The set $\mathcal{K}_s$ satisfies the following properties:

\begin{prop} \label{conjuntok}
	$\mathcal{K}_s$ is a convex, closed and bounded subset of $H_0^s(\O)$.  
	\end{prop}
	
\begin{proof}
	Clearly, $\mathcal{K}_s$ is a convex set. $\mathcal{K}_s$ is also bounded. Indeed, given $u\in \mathcal{K}_s$, by H\"older and Poincar\'e's inequalities we get
	\begin{align*}
		c(n,s)[u]_s^2 \leq \int_{\O} u \, dx 
		\leq |\O|^{\frac{1}{2}} \|u\|_{L^2(\O)} \leq C|\O|^{\frac{1}{2}}[u]_s.
	\end{align*} 

	In order to see that $\mathcal{K}_s$ is closed, let $\{u_k\}_{k\in \N}$ be a sequence in $\mathcal{K}_s$ such that $u_k\to u$ in $H_0^s(\O)$. For any $k\in \N$ and any $v\in H_0^s(\O)$, $v\ge 0$, it holds that
	$$
		\E(u_k,v)\leq \int_{\O}{v \, dx}.
	$$
	Taking the limit as $k\to \infty$ we obtain that $\E(u,v)\leq \int_{\O}{v \,  dx}$, but,  since $v\in H^s_0(\Omega)$ is nonnegative but otherwise arbitrary we obtain that  $(-\Delta)^s u \leq 1$ in $\O$ and then $u \in \mathcal{K}_s$.
\end{proof}

\begin{remark}\label{cota.ks}
Observe that optimal constant in Poincar\'e's inequality 
$$
\|u\|_{L^2(\O)}^2\le C(\Omega,s)[u]_{s}^2,
$$
has a dependence on $s$ of the form
$$
C(\Omega,s)\le (1-s) C(\Omega).
$$
See \cite{BPS}.

Therefore, the proof of Proposition \ref{conjuntok} gives that if $u\in \K_s$, then
$$
(1-s)[u]_s^2\le C,
$$
where $C$ depends on $\Omega$ but is independent on $0<s<1$.
\end{remark}

Now, in order to prove the existence of solution to \eqref{problem} we define a functional $G_s$ on $\mathcal{K}_s$ satisfying the conditions $(G_1)$--$(G_3)$.

We will use the notation, for $0<s\le 1$,
\begin{equation}\label{Asc}
\A_s^c(\Omega) := \{ A\in \A_s(\Omega)\colon |A|\le c\}.
\end{equation}

For any $0<s\le 1$, given $w\in \mathcal{K}_s$ we define 
\begin{equation} \label{J}
	J_s(w)=\inf \{F_s(A) \colon A\in \A_s^c(\Omega),\ u_A^s \leq w \}.
\end{equation}
This functional $J_s$ is not lower semicontinuous in general. So we define $G_s$ to be the lower semicontinuos envelope of $J_s$ in $\mathcal{K}_s$ with respect to the strong topology in $L^2(\Omega)$, i.e., 
\begin{equation} \label{G}
G_s(w)=\inf \{\liminf_{k\to\infty} J_s(w_k)\},
\end{equation}
where the infimun is taken over all sequences $\{w_k\}_{k\in \N}$ in $\mathcal{K}_s$ such that $w_k \to w$ in $L^2(\Omega)$.

Observe that $G_s$ automatically verifies ($G_2$).

\begin{prop} Let $0<s\le1$. The functional $G_s$ satisfies conditions $(G_1)$.
\end{prop}

\begin{proof}
The case $s=1$ is considered in \cite{Buttazzo-DalMaso}, so we take $0<s<1$.

Let us begin by noticing that if $u,v\in \K_s$, then $\max\{u,v\}\in \K_s$. In fact, let us denote $w= \max \{u, v\}$ and consider the convex set 
$$
E=\{z \in H^s_0(\Omega) \colon z \leq w \textrm{ in } \Omega \}.
$$
By Stampacchia's Theorem  there exists a unique function $z_0 \in E$ such that 
$$
J(z_0):=\frac{c(n,s)}{2}[z_0]_s^2 - \int_{\Omega} z_0\, dx = \min_{E}J.
$$
In addition, $z_0$ satisfies that
\begin{equation} \label{zdv}
\mathcal{E}(z_0,z-z_0) \geq \int_{\Omega}{(z-z_0) \, dx} \quad \mbox{ for all }z \in E,
\end{equation}
where $\E$ is defined in \eqref{E.def}.
	
Let us see that $(-\Delta)^s z_0 \leq 1$. Given $\vp \in H^s_0(\Omega)$ such that $\vp \leq 0$, we define the functional 
$$
i(t)=J(z_0 + t\vp), \  \textrm{ for all }t\geq 0.
$$
Observe that $i'(0)\geq0$. In consequence, for any non-positive $\vp \in H^s_0(\Omega)$ it holds that $\mathcal{E}(z_0,\varphi) \geq\int_{\Omega}\vp \,dx$, and then  $\mathcal{E}(z,\varphi) \leq \int_{\Omega}\vp \,dx$ for any $\vp \in H^s_0(\Omega)$, $\vp \geq0$ and the claim follows.

Now, we will prove that $z_0\geq u$ (and for symmetry reasons that $z_0\geq v$),  from where it will follow that $z_0\geq w$. Since $z_0\in E$, the reverse inequality holds and we can conclude that $z_0=w\in \K_s$.
	
Let $\eta= \max \{z_0, u\}$ and let us see that $z_0=\eta$. Observe that $\eta \in E$ and thus it can be consider as a test function in \eqref{zdv}. Thus,
$$
\mathcal{E}(z_0,\eta-z_0) \geq \int_{\Omega} (\eta-z_0) \,dx.
$$
On the other hand, since $\eta-z_0 \geq0$ and $(-\Delta)^s u\le 1$ in $\Omega$, it follows that
$$
\mathcal{E}(u,\eta-z_0) \leq \int_{\Omega}{(\eta-z_0) \, dx}.
$$
From  both inequalities it is straightforward to see that 
$$
0\leq \mathcal{E}(z_0-u,\eta-z_0) \leq -c(n,s) [(u-z_0)^+]_{s}^2
$$
and then $(u-z_0)^+=0$ in $\R^n$, which implies that $z_0\geq u$ in $\R^n$, as we required.

Now we proceed with the proof.

Let $u,v\in\mathcal{K}_s$ be such that $u\leq v$ and let $\{u_k\}_{k\in\N}\subset \K_s$ be such that  $u_k\to u$ in $L^2(\O)$ and $J_s(u_k)\to G_s(u)$.

By our previous claim, $v_k=\max\{ v, u_k\}\in \mathcal{K}_s$ for each  $k\in \N$ and $v_k \to v= \max\{ v, u\}$ in $L^2(\O)$. Consequently, since $J_s$ is non-increasing   and $v_k\geq u_k$ for any $k\in \N$ we get
	$$
	G_s(v)\leq \liminf_{k\to \infty} J_s(v_k) \leq \lim_{k\to \infty}J_s(u_k)=G_s(u),
	$$
as we wanted to show.
\end{proof}

We will need the following lemma in order to prove condition ($G_3$). We omit the proof since it is completely analogous to   that of Lemmas 3.2 and 3.3 in \cite{Buttazzo-DalMaso} where the case $s=1$ was considered.

\begin{lema} \label{lemaepsilon}
Let $0<s<1$. Let $\{A_k\}_{k\in\N}\subset \A_s(\O)$ be such that $u_{A_k}^s\to u$ in $L^2(\O)$, with $u\leq u_A^s$. We define $A^{\ve}=\{ u_A^s>\ve\}$. Then, if $u_{A_k\cup A^{\ve}}^s \to u^{\ve}$ in $L^2(\O)$ it holds that $u^{\ve} \leq u_A^s$.
\end{lema}

With the help of Lemma \ref{lemaepsilon} we are able to show that $G_s$ satisfies condition ($G_3$).

\begin{prop} \label{G3}
Let $0<s\le 1$. Then the functional $G_s$ satisfies ($G_3$). 
\end{prop}

\begin{proof} 
We only need to consider $0<s<1$. Let us fix  $A\in\A_s^c(\Omega)$. From \eqref{G} it follows that $G_s(u_A^s)\leq F_s(A)$. To prove the reverse inequality, it suffices to see that 
$$
F_s(A)\leq \liminf_{k\to  \infty}J_s(w_k)
$$
for any sequence $\{w_k\}_{k\in \N}\subset \mathcal{K}_s$ such that $w_k\to u_A^s$ in $L^2(\O)$. 
	
By definition of $J_s$, there exists $A_k\in\A_s^c(\Omega)$ such that 
\begin{equation} \label{cotafinal}
F_s(A_k) \leq J_s(w_k) +\frac{1}{k}\qquad  \mbox{and} \qquad u_{A_k}^s\leq w_k.
\end{equation}  
	
Observe that $u_{A_k}^s\in \mathcal{K}_s$ for each $k\in \N$ and by Proposition \ref{conjuntok}, $\{u_{A_k}^s\}_{k\in \N}$ is  bounded   in $H_0^s(\O)$. Then, up to a subsequence, there exists $u\in \mathcal{K}_s$ such that $u_{A_k}^s \to u$ in $L^2(\O)$. Since $w_k\to u_A^s$  in $L^2(\O)$, from  $u_{A_k}^s\leq w_k$ we get  $u\leq u_A^s$.
	
Let us consider the set $A^{\ve}=\{ u_A^s> \ve \}$ and observe that $u_{A_k \cup A^{\ve}}^s \in \mathcal{K}_s$. Again by Proposition \ref{conjuntok}, it follows that  and $u_{A_k\cup A^{\ve}}^s \to u^{\ve}$ in $L^2(\O)$ for some $u^{\ve}\in \mathcal{K}_s$. By Lemma \ref{lemaepsilon}, inequality $u^{\ve}\leq u_A^s$ follows.
	
We claim that $(u_A^s -\ve)^+\le u_{A^\ve}^s$. Indeed,
$$
(u_A^s-\ve)^{+}(x)-(u_A^s-\ve)^{+}(y)= 
\begin{cases}
u_A^s(x)-u_A^s(y) & \text{ if } x, y \in A^\ve \\
u_A^s(x)-\ve & \text{ if } x\in A^\ve \text{ and } y\not\in A^\ve \\
-u_A^s(y)+\ve & \text{ if } x\not\in A^\ve \text{ and } y\in A^\ve \\
0 &\text{ otherwise.}
\end{cases}
$$

Then, for any $v\in H_0^s(A^{\ve})$ such that $v\ge 0$, we get
\begin{align*}
&\iint_{\R^n\times\R^n} \frac{( (u_A^s(x)-\ve)^+ - ((u_A^s(y)-\ve)^+)(v(x)-v(y))}{|x-y|^{n+2s}}\, dxdy = \\
& \iint_{A^\ve \times A^\ve} \frac{(u_A^s(x) - u_A^s(y))(v(x)-v(y))}{|x-y|^{n+2s}}\, dxdy + 2\iint_{A^\ve \times (A^\ve)^c} \frac{(u_A^s(x)-\ve) v(x)}{|x-y|^{n+2s}}\, dydx =\\
& \iint_{\R^n\times\R^n} \frac{(u_A^s(x) - u_A^s(y))(v(x)-v(y))}{|x-y|^{n+2s}}\, dxdy + 2 \iint_{A^\ve \times (A^\ve)^c} \frac{(u_A^s(y)-\ve) v(x)}{|x-y|^{n+2s}}\, dydx \le\\
&  \iint_{\R^n\times\R^n} \frac{(u_A^s(x) - u_A^s(y))(v(x)-v(y))}{|x-y|^{n+2s}}\, dxdy 
\end{align*}

That is,  $(-\Delta)^s (u_A^s-\ve)^+ \le (-\Delta)^s u_A^s = 1 = (-\Delta^s) u_{A^\ve}^s$ in $A^\ve$. Moreover, since $0=(u_A^s-\ve)^+ = u_{A^\ve}^s$ in $\R^n \setminus A^\ve$, from the comparison principle  it follows that $(u_A^s-\ve)^+ \le u_{A^\ve}^s$ in $\R^n$.
	
	We have obtained the following chain of inequalities
	$$
		(u_A^s-\ve)^+ \le u_{A^\ve}^s\le u_{A_k\cup A^\ve}^s.
	$$
 Taking limit as $k\to\infty$ we conclude that
	$$
		(u_A^s-\ve)^+ \le u^\ve\le u_A^s.
	$$
	since $u^\ve \le u_A^s$ and $u_{A_k\cup A^\ve}^s \to u^\ve$.
	
Since $u^\ve \in \mathcal{K}_s$, $\{u^\ve\}_{\ve>0}$ is uniformly bounded in $H_0^s(\O)$. Consequently, up to a subsequence, $u^\ve\to u_A^s\in L^2(\Omega)$. 

By a standard diagonal argument, there exists a sequence $\{\ve_k\}_{k\in\N}$ such that $u_{A_k\cup A^{\ve_k}}^s\to u_A^s$ in $L^2(\O)$.
	 
In conclusion, we have proved that $(A_k\cup A^{\ve_k})$ $\gamma_s-$converges to $A$. Therefore
$$
F_s(A)\le \liminf_{k\to\infty} F_s(A_k\cup A^{\ve_k}) \le \liminf_{k\to\infty} F_s(A_k) \le \liminf_{k\to\infty} J_s(w_k).
$$
This fact concludes the proof of the proposition.
\end{proof}

Having proved these preliminary results, the proof of Theorem \ref{main.s} follows 	in the same way of that of \cite[Theorem 2.5]{Buttazzo-DalMaso}. We include the details for the reader convenience.

\begin{proof}[Proof of Theorem \ref{main.s}]
First, we solve \eqref{minG}. Take $\{w_k\}_{k\in \N} \subset \K_s$ be such that $|\{w_k>0\}|\le c$ and 
$$
\lim_{k\to \infty}G_s(w_k)= \inf\left\{ G_s(w) \colon w\in \K_s, \, |\{ w>0 \}|\le c  \right\}=:m_{G_s}.
$$

By Proposition \ref{conjuntok}, there exists $w_0 \in \K_s$ such that $w_k\to w_0$ strongly in $L^2(\O)$, up to a subsequence. Thus, $|\{w_0 >0\}|\le c$. Then, by $(G_2)$,
$$
m_{G_s}\le G_s(w_0)\le \liminf_{k\to\infty} G_s(w_k)=m_{G_s}.
$$
So, $w_0$ is a solution to \eqref{minG}.

Now, consider $A_0:=\{ w_0 >0 \}$. Then, $A_0\in \A_s^c(\O)$. By Lemma \ref{teouA}, $w_0 \le u_{A_0}^s$.

For every $A\in \A_s^c(\O)$, we konw that $u_A^s\in \K_s$, $|\{u_A^s>0\}|\le c$. Then, by $(G_3)$,$(G_1)$ and the fact that $w_0$ is the solution to \eqref{minG}, we have
$$
F_s(A_0)=G_s(u_{A_0}^s)\le G_s(w_0)\le G_s(u_A^s)=F_s(A).
$$

Therefore, $A_0$ is a solution to \eqref{problem}.
\end{proof}

\section{Proof of Theorem \ref{main}} \label{sec.teo2}
 In this section we prove Theorem \ref{main} following the same spirit of \cite{Buttazzo-DalMaso}, however, nontrivial changes must be performed due to the nonlocal settings.

Our first goal is to show that a sequence $\{u_k\}_{k\in\N}\subset L^2(\O)$ such that $u_k\in \K_{s_k}$  is precompact and that every accumulation point belongs to $\K_1$.

This is the content of the next lemma.

\begin{lema} \label{Ksacotado}
Let $0<s_k\uparrow 1$ and let $u_k\in \K_{s_k}$. Then, there exists $u\in H^1_0(\O)$  and a subsequence $\{u_{k_j}\}_{j\in \N}\subset\{u_k\}_{k\in\N}$ such that $u_{k_j} \to u$ strongly in $L^2(\O)$. 
	
Moreover, if $u_k \in \K_{s_k}$ is such that $u_k\to u$ strongly in $L^2(\O)$, then $u\in \K_1$.
\end{lema}

\begin{proof}
From Remark \ref{cota.ks}, there exists a constant $C>0$ such that
$$
\sup_{k\in\N}(1-s_k)[u_k]_{s_k}^2\le C.
$$
Now the first claim follows from \cite[Theorem 4]{BBM}.

Now, assume that $u_k\to u$ in $L^2(\O)$. It is clear that $u\geq 0$. Since $(-\Delta)^{s_k} u_k\leq 1$ in $\Omega$, for every non-negative $\vp \in C_c^{\infty}(\O)$ we have that
$$
\int_\Omega (-\Delta )^{s_k} \varphi u_k \, dx  = \langle (-\Delta )^{s_k} u_k, \varphi \rangle  \leq \int_{\O} \vp \, dx.
$$

By the convergence assumption on $u_k$ and the fact that the convergence \eqref{conv.lap} is also strong in $L^2(\Omega)$, we can take limit as $k\to \infty$ in the previous inequality to obtain that
$$
\int_\Omega -\Delta \vp u\, dx  = \langle -\Delta  u , \vp \rangle \leq \int_{\O} \vp \, dx,
$$
and conclude that   $-\Delta u \leq 1$ in $\O$. Consequently, $u\in \K_1$ as required.
\end{proof}

Analogously as in the previous section, we define the following functionals for the limit problem
$$
J_1(w):=\inf \left\{ F_1(A)\colon A\in \A_1^c(\Omega),\  u_A\le w \right\},
$$
where $\A_1^c(\O)$ is defined in \eqref{Asc} for $s=1$ and we define $G_1$ to be the lower semicontinuous envolpe of $J_1$ in $\K_1$.

The next lemma gives the continuity of $u_A^s$ when $s\uparrow 1$.
\begin{lema}\label{Afijo}
	For every $A \in \A_1(\Omega) $, $u_A^s\to u_A$ strongly in $L^2(\Omega)$, when $s\uparrow 1$.
\end{lema}
\begin{proof}
	Let us remind that, from Lemma \ref{igualdad}, $u_A^s$ is also the solution to the minimization problem 
	$$
	I_s(u_A^s)=\min \{ I_s(w) \colon  w\in L^2(\O)\},
	$$
	where
	$$
	I_s(w)=\begin{cases}\begin{array}{ll}
	\frac{c(n,s)}{2}[w]_s^2-\int_{\O} w \, dx & \textrm{ if } w\in H_0^s(A), \\
	\infty & \textrm{ otherwise.}
	\end{array}
	\end{cases}
	$$
	Notice that, by \cite{Ponce}, we have that $\frac{c(n,s)}{2}[w]_s^2 \stackrel{\Gamma}{\to} \frac{1}{2}\| \nabla w \|_2^2$.	Since the $\Gamma$-convergence is stable under continuous perturbations, we have that $I_s \stackrel{\Gamma}{\to}  I_1$ in $L^2(\O)$, where
	$$
	I_1(w)=\begin{cases}\begin{array}{ll}
	\frac{1}{2}\| \nabla w \|_2^2-\int_{\O} w \, dx & \textrm{ if } w\in H_0^1(A), \\
	\infty & \textrm{ otherwise.}
	\end{array}
	\end{cases}
	$$
	Thus, the minimizer of $I_s$ converges to the minimizer of $I_1$. That is $u_A^s\to u_A$ strongly in $L^2(\O)$.
\end{proof}

Now we address the more difficult problem of understanding the limit behavior of $u_A^s$ when the domains also are varying with $s$.

This first lemma is key in understanding this limit behavior and the ideas are taken from \cite{Buttazzo-DalMaso}.
\begin{lema} \label{ELlemask}
	Let $0<s_k\uparrow 1$ and for every $k\in \N$ let  $A_k\in \A_{s_k}(\Omega)$ be such that $u_{A_k}^{s_k} \to u$ strongly in $L^2(\O)$. Let $\{w_k\}_{k\in \N}\subset L^2(\O)$ be such that $w_k \in H_0^{s_k}(A_k)$ for every $k\in \N$ and $\sup_{k\in \N}(1-s_k)[w_k]^2_{s_k}<\infty$. Assume, moreover that $w_k\to w$ strongly in $L^2(\O)$. Then, $w\in H_0^1(\{u>0\})$.
\end{lema}

\begin{proof}
We need to show that $w=0$ in $\R^n \setminus \{u>0\}$, i.e., $w=0$ in $\{u=0 \}$.
	
Let us define the functional 
	\begin{equation} \label{phi}
	\Phi_k(v)= 
	\begin{cases}
	\begin{array}{ll}
	\frac{c(n,s_k)}{2}[v]_{s_k}^2 &\textrm{ if } v\in H_0^{s_k}(A_k), \\
	+\infty &\textrm{ otherwise}.
	\end{array}
	\end{cases}
	\end{equation}
	defined in $L^2(\Omega)$. By the compactness of $\Gamma$-convergence, there exists a subsequence still denote by $\Phi_k$ such that 
	$$
	\Phi_k \stackrel{\Gamma}{\to} \Phi \quad \mbox{ in }L^2(\Omega).
	$$ 	
	From \cite[Theorem 11.10]{DalMaso},  $\Phi$ is  a quadratic form in $L^2(\Omega)$ with domain $D(\Phi) \subset L^2(\O)$.

	Observe that $w\in D(\Phi)$, since
	$$
	\Phi(w)\leq \liminf_{k\to +\infty} \Phi_k(w_k)\le \sup_{k\in\N} \frac{c(n,s_k)}{2}[w_k]_{s_k}^2 \le C\sup_{k\in\N} (1-s_k) [w_k]_{s_k}^2 <  \infty.
	$$ 
	
	Let $B\colon D(\Phi)\times D(\Phi)\to \R$ be the bilinear form associeted to $\Phi$, which is defined by
	$$
	B(v,\eta)=\frac{1}{4}(\Phi(v+\eta)-\Phi(v-\eta)).
	$$
	Let us denote by $V$ the closure of $D(\Phi)$ in $L^2(\Omega)$ and consider the linear operator  $T\colon D(T)\subset L^2(\Omega) \to L^2(\Omega)$ defined as $Tv=f$ where
	$$
	D(T)=\left\{ v\in D(\Phi) \colon \exists f \in V \textrm{ such that } B(v,\eta)=\int_\O{f\eta \, dx}, \  \forall \eta \in D(\Phi)\right\}.
	$$	
	By \cite[Proposition 12.17]{DalMaso}, $D(T)$ is dense in $D(\Phi)$ with respect to the norm 
	$$
	\|v\|_{\Phi}=(\|v\|_{L^2(\O)}+\Phi(v))^{\frac{1}{2}}.
	$$
	Moreover,   the following relation holds
	\begin{equation} \label{relac}
	\sqrt{2}\| \cdot \|_{\Phi}\geq \| \cdot \|_{H_0^1(\O)}.
	\end{equation}
	Indeed, 	if $z\in D(\Phi)$, as $\Phi_k \stackrel{\Gamma}{\rightarrow} \Phi$ in $L^2(\Omega)$, there exists $\{z_k\}_{k\in \N}$ such that $z_k\to z$ in $L^2(\Omega)$ and 
	$$
	\infty>\Phi(z)=\lim_{k\to\infty}{\Phi_k(z_k)}=
	\begin{cases}
	\begin{array}{ll}
	\lim_{k\to\infty} \frac{c(n,s_k)}{2}[z_k]_{s_k}^2 &\textrm{ if } z_k\in H_0^{s_k}(A_k), \\
	\infty &\textrm{ otherwise. }
	\end{array}
	\end{cases}
	$$  
	Thus, $z_k\in H_0^{s_k}(A_k)$ and then 
	$$
	\|z\|_{H^1_0(\O)}^2 \leq \liminf_{k\to \infty} c(n,s_k) [z_k]_{s_k}^2 = 2 \lim_{k\to \infty}\Phi_k(z_k) = 2\Phi(z) \leq 2 \|z\|_\Phi^2.
	$$

	Since \eqref{relac} holds, $D(T)$ is dense in $D(\Phi)$ with respect to  the strong topology of $H^1_0(\Omega)$. Now to achieve the proof it is enough to prove that $v=0$ in $\{ u=0\}$ for all $v\in D(T)$.

	Let $v\in D(T)$ and let $f\in Tv$; then $v$ is a minimum point of the functionl
	$$
	\Psi(\eta)=\frac{1}{2}\Phi(\eta)-\int_{\O}{f\eta \, dx}
	$$
	(see, \cite[Proposition 12.12]{DalMaso}).	Let $v_k$ be the minimum point of functional 
	$$
	\Psi_k(\eta):=\frac{1}{2}\Phi_k(\eta)-\int_{\O}{f\eta \, dx};
	$$
	then $v_k$ is the solution of the problem
	$$
	(-\Delta)^{s_k} v_k=f, \qquad v\in H^{s_k}_0(A_k).
	$$
	Since $\Phi_k \stackrel{\Gamma}{\rightarrow} \Phi$, then $\Psi_k \stackrel{\Gamma}{\rightarrow} \Psi$ and so we have that $v_k\to v$ strongly in $L^2(\Omega)$.
	
	For $\ve>0$ we consider $f^{\ve}$ to be a bounded function with compact support such that $\|f^{\ve}-f\|_2 <\ve$ and  $v_k^{\ve}$ is  solution of 
	$$
	(-\Delta)^{s_k} v_k^{\ve}=f^{\ve} \textrm{ in } A_k, \qquad v_k^{\ve} \in H_0^{s_k}(A_k).
	$$ 
 
	By using the linearity of the operator together with H\"older's and Poincar\'e's inequalities we get
	\begin{align*}
		\frac{c(n,s_k)}{2}[v^{\ve}_k-v_k]^2_{s_k} 
		%c(n,s)([v^\ve_k]_{s_k} + [v_k]_{s_k} -2[v^\ve_k]_{s_k} [v_k]_{s_k})\\
		%&=2\int_\Omega f^\ve v_k^\ve\,dx  + 2\int_\Omega f v_k\,dx - 4\E(v_k^\ve,v_k)\\
		%&=2\int_\Omega f^\ve v_k^\ve\,dx  + 2\int_\Omega f v_k\,dx - 2\int_\Omega f^\ve v_k\,dx - 2\int_\Omega f v_k^\ve\,dx\\
		&= \int_\Omega (f^\ve-f)(v_k^\ve-v_k)\,dx\\
		& \leq  \|f_\ve-f\|_2 \|v_k^\ve-v_k\|_2.
	\end{align*}
	From Poincar\'e's inequality we obtain that
	$$
		(1-s_k)[v^{\ve}_k-v_k]^2_{s_k} \leq C \ve^2,
	$$
	where $C$ is independent on $k$.

	Then, from \cite[Theorem 4]{BBM}, up to a subsequence, $v_k^\ve \to v^\ve$ strongly in $L^2(\Omega)$ and $\|v^\ve-v\|_{H^1_0(\O)}\leq C\ve$.	At this point is enough to prove that $v^{\ve}=0$ in $\{u=0\}$ for all $\ve>0$.

	Since $f^{\ve} \leq c^{\ve}:=\|f^{\ve}\|_{\infty}$ and 
	$$
	(-\Delta)^{s_k} v^{\ve}_k=f^{\ve}\leq c^{\ve}=(-\Delta)^{s_k}(c^{\ve}u_{A_k}^{s_k})\textrm{ in } A_k, \qquad
	v_k^{\ve}= c^{\ve}u_{A_k}^{s_k}=0 \textrm{ in } \R^n \setminus A_k,
	$$
	the comparison principle gives that $v^{\ve}_k\leq c^{\ve}u_{A_k}^{s_k}$.	Analogously,
	$-v^\ve_k\le c^\ve u_{A_k}^{s_k}$.
	
	As $k\to \infty$, we obtain that $|v^{\ve}|\leq c^\ve u$, which implies that $v^{\ve}=0$ in $\{u=0\}$ for any $\ve>0$ and that completes the proof.
\end{proof}

The next lemma is the counterpart of Lemma \ref{lemaepsilon} of the previous section. We include here the details for completeness. The main modifications with respect to the previous proof (which was analogous to that of Lemmas 3.2 and 3.3 in \cite{Buttazzo-DalMaso}) where carried out in the previous Lemmas of this section.
\begin{lema} \label{lemaepsilonsk}
Let $0<s_k\uparrow 1$ and for every $k\in\N$,  let $A_k\in \mathcal{A}_{s_k}(\Omega)$, $A \in \mathcal{A}_1(\Omega)$.  Assume that $u_{A_k}^{s_k}\to u$ in $L^2(\O)$ and that $u\leq u_A$.

Then, if $u_{A_k\cup A^{\ve}}^{s_k} \to u^{\ve}$ strongly in $L^2(\O)$, where $A^{\ve}:=\{ u_A>\ve\}$, it holds that $u^{\ve} \leq u_A$.
\end{lema}

\begin{proof}
	By Lemma \ref{teouA} with $s=1$, the inequality $u^{\ve}\leq u_A$ will follow if we prove that   $u^\ve\in H^1_0(\Omega)$, $u^{\ve}\leq 0$ in $\R^n \setminus A$ and  $-\Delta u^{\ve}\leq 1$ in  $\O$. 
	
	Observe that by Lemma \ref{Ksacotado} we have that $u,u^\ve \in H_0^1(\O)$.
	Let us define
	$$
	v^{\ve}:=1-\tfrac{1}{\ve} \min\{u_A, \ve\}=\tfrac{1}{\ve}(\ve - u_A)^+.
	$$
	and observe that $0\leq v^{\ve} \leq 1$ and  $v^{\ve}=0$ in $A^{\ve}$ 	since $0\leq \min \{u_A, \ve\} \leq \ve$ and $\frac{1}{\ve} \min \{ u_A , \ve \}=1$ in $A^{\ve}$. 	If we define
	$$
	u_{k,\ve}:=u_{A_k\cup A^{\ve}}^{s_k}, \qquad w_{k,\ve}:= \min \{ v^{\ve}, u_{k,\ve}  \},
	$$
	it holds that  $w_{k,\ve}\geq 0$ since the comparison principle gives $u_{k,\ve}\geq 0$, and also $v^{\ve}\geq 0$.
	
	Since $v^{\ve}=0$ in $A^{\ve}$, it follows that $w_{k,\ve}=0$ in $A^{\ve}$. Moreover, since $u_{k,\ve}=0$ in $\R^n\setminus (A_k \cup A^{\ve})$, it holds that $w_{k,\ve}=0$ in $\R^n \setminus (A_k \cup A^{\ve})$, and consequently, $w_{k,\ve} \in H_0^{s_k}(A_k)$. 
	%To be clear, we are saying 
	%$$
%	w_{k,\ve}=0 \textrm{ in } A^{\ve} \cup (\R^n \setminus (A_k \cup A^{\ve}))=A^{\ve} \cup (\R^n \setminus A_k) \supset \R^n \setminus A_k.
%	$$
	
	Notice that $w_{k,\ve} \to w_\ve:=\min \{v^{\ve}, u^{\ve}\}$ strongly in $L^2(\O)$, and then, applying Lemma \ref{ELlemask}, we get $w_\ve \in H_0^1(\{u>0\})$, from where $w_\ve=0$ in $\{u=0\}$. The relation $0\leq u \leq u_A$ implies the  inclusion $\{u_A=0\} \subset \{u=0\}$, from where  $w_\ve \in H_0^1(\{u_A>0\})$. 	Moreover, since $\{u_A>0\} \subset A$, we have that $w_\ve=0$ in $\R^n \setminus A$. Now, being $v^{\ve}=1$ in $\R^n \setminus A$, we get $u^{\ve}=0$ in $\R^n \setminus A$, and in particular, $u^{\ve}\leq 0$ in $\R^n \setminus A$.
	
	Finally, it remains to see that $-\Delta u^{\ve}\leq 1$ in $\O$. Observe  that $u_{k,\ve} \in \K_{s_k}$ and $u_{k,\ve}\to u^{\ve}$ strongly in $L^2(\O)$. Then  $u^{\ve} \in \K_1$ by Lemma \ref{Ksacotado}. Thus  $-\Delta u^{\ve} \leq 1$ in $\O$ and the proof is complete.
\end{proof}

With the help of these lemmas, we are now in position to prove then main tool needed in the proof of Theorem \ref{main}.

\begin{prop} \label{suplente}
	Let $0<s_k\uparrow1, A_k \in \A_{s_k}^c(\O)$ be such that $u_{A_k}^{s_k}\to u$ strongly in $L^2(\O)$. Then, there exist $\tilde{A}_k \in \A_{s_k}(\O)$ such that $A_k \subset \tilde{A}_k$ and $\tilde{A}_k$ $\gamma-$converges to $A:=\{ u>0\}$.
\end{prop}

\begin{proof}
Since $u_{A_k}^{s_k} \in \K_{s_k}$ and $u_{A_k}^{s_k}\to u$, by Lemma \ref{Ksacotado}, $u\in \K_1$. Then, by Lemma \ref{teouA}, $u\le u_A$.
	
As in the previous proof, consider $A^\ve := \{ u_A > \ve \}$ and observe that
\begin{equation*} 
u_{A^\ve}^{s_k} \le u_{A_k\cup A^\ve}^{s_k}
\end{equation*}
	
Since $u_{A_k\cup A^\ve}^{s_k} \in \K_{s_k}$, by Lemma \ref{Ksacotado}, there exists $u^\ve\in H^1_0(\Omega)$ such that $u_{A_k\cup A^\ve}^{s_k} \to u^\ve$ strongly in $L^2(\O)$, up to a subsequence.
	
Also, by Lemma \ref{Afijo}, $u_{A^\ve}^{s_k} \to u_{A^\ve}$ strongly in $L^2(\O)$. Then, we can pass to the limit as $k\to \infty$ in the previous inequality to conclude that
$$
u_{A^\ve} \le u^{\ve}
$$
It can be easily checked that $u_{A^\ve}=(u_A-\ve)_+$.	Moreover,  from Lemma \ref{lemaepsilonsk},
$$
(u_A-\ve)_+ \leq u^{\ve} \leq u_A.
$$
Thus, there exists a sequence $0<\ve_k\downarrow 0$ such that
$$
u_{A_k\cup A^{\ve_k}}^{s_k} \to u_A \textrm{ strongly in } L^2(\O).
$$
That is, $A_k\cup A^{\ve_k}=:\tilde{A_k} \,  \gamma$-converges to $A$.
\end{proof}

Now we are ready to prove the main result.

\begin{proof}[Proof of Theorem \ref{main}]
By Theorem \ref{main.s}, there exists $A_k\in \A_{s_k}^c(\Omega)$ such that
$$
F_{s_k}(A_k) = \min\{ F_{s_k}(A)\colon A\in\A_{s_k}^c(\Omega)\}.
$$
Then, if $A\in \A_1^c(\Omega)$, by condition $(H_1)$ we know that
$$
\limsup_{k\to \infty} F_{s_k}(A_k)\le \lim_{k\to \infty} F_{s_k}(A) = F_1(A),
$$
from where it follows that
\begin{equation}\label{abajo}
\limsup_{k\to\infty} \min\{ F_{s_k}(A)\colon A\in\A_{s_k}^c(\Omega)\} \le  \min\{ F_1(A)\colon A\in\A_1^c(\Omega)\}.
\end{equation}

Let us see the reverse inequality. By simplicity, let us denote $u_k:=u_{A_{k}}^{s_k} \in \K_{s_k}$. 
		
By Lemma \ref{Ksacotado}, there is $u\in H^1_0(\Omega)$ such that, up to a subsequence, $u_k \to u$ strongly in $L^2(\O)$. 
	
Moreover, by Proposition \ref{suplente}, there exists $\tilde{A_k} \in \A_{s_k}(\O)$ such that $A_k \subset \tilde{A}_k$ and $\tilde{A}_k$  $\gamma-$converges to $A:=\{ u>0 \}$. Since $u_k \to u$ in $L^2(\O)$, $|A|\le c$.

Finally, from  condition $(H_2)$ and $(H_2^s)$ we   conclude that
$$
F_1(A)\leq \liminf_{k\to \infty} F_{s_k}(\tilde{A}_{k}) \leq \liminf_{k\to \infty} F_{s_k}({A}_{k}),
$$
from where it follows that
\begin{equation}\label{arriba}
\min\{ F_1(A)\colon A\in\A_1^c(\Omega)\}\le \liminf_{k\to\infty} \min\{ F_{s_k}(A)\colon A\in\A_{s_k}^c(\Omega)\}.
\end{equation}
Putting together \eqref{abajo} and \eqref{arriba} the result follows.
\end{proof}

%%%%%%%%%%%%%%%%%%%%%%%%%%%%%
%%%%%%%%%%%%%%%%%%%%%%%%%%%%%
%% 
%% ACKNOWLEDGEMENTS
%%
%%%%%%%%%%%%%%%%%%%%%%%%%%%%%
%%%%%%%%%%%%%%%%%%%%%%%%%%%%%

\section*{Acknowledgements}
This paper was partially supported by grants UBACyT 20020130100283BA, CONICET PIP 11220150100032CO and ANPCyT PICT 2012-0153. 

J. Fern\'andez Bonder and Ariel M. Salort are members of CONICET and A. Ritorto is a doctoral fellow of CONICET.

%%%%%%%%%%%%%%%%%%%%%%%%%%%%%
%%%%%%%%%%%%%%%%%%%%%%%%%%%%%
%% 
%% BIBLIOGRAFIA
%%
%%%%%%%%%%%%%%%%%%%%%%%%%%%%%
%%%%%%%%%%%%%%%%%%%%%%%%%%%%%
\bibliographystyle{amsplain}
\bibliography{biblio}

\end{document}